\documentclass[12pt]{amsart}

\usepackage{bm,amssymb,color}
\usepackage{graphicx}

\numberwithin{equation}{section}

\newtheorem{theorem}{Theorem}[section]
\newtheorem{proposition}[theorem]{Proposition}
\newtheorem{lemma}[theorem]{Lemma}
\newtheorem{corollary}[theorem]{Corollary}

\theoremstyle{definition}

\newtheorem{example}[theorem]{Example}

\newtheorem{remark}[theorem]{Remark}

\newcommand{\Tor}{\ensuremath{\mathrm{Tor}}\hspace{1pt}}

\def\cocoa{{\hbox{\rm C\kern-.13em o\kern-.07em C\kern-.13em o\kern-.15em A}}}

\begin{document}

\title[Betti splittings and cover ideals of bipartite graphs]{Betti splittings and multigraded Betti numbers of cover ideals of bipartite graphs}

\author{Satoshi Murai}
\address{
Satoshi Murai,
Department of Mathematics
Faculty of Education
Waseda University,
1-6-1 Nishi-Waseda, Shinjuku, Tokyo 169-8050, Japan}
\email{s-murai@waseda.jp}
\thanks{The first author is partially supported by KAKENHI 21K03190 and 21H00975.
}

\author{Mitsuki Shiina}
\address{
Mitsuki Shiina,
Department of Mathematics
Faculty of Education
Waseda University,
1-6-1 Nishi-Waseda, Shinjuku, Tokyo 169-8050, Japan}
\email{mikkii\_shiina@akane.waseda.jp}

%\author{Another Author}
%\address{
%}
%\email{}

%Keyword and Subject Classes (if needed)
%\keywords{}
%\subjclass[2000]{}
%\dedicatory{Dedicated to on the occasion of his birthday}

\begin{abstract}
In this note, we study Betti splittings
 of cover ideals of bipartite graphs.
We prove that if $J \subset \Bbbk [x_1,\dots,x_n]$ is the cover ideal of a bipartite graph then the $x_i$-partition of $J$ is a Betti splitting for any $i$.
We also prove that
multigraded Betti numbers of any squarefree monomial ideal can appear in a certain part of multigraded Betti numbers of the cover ideal of a bipartite graph.
\end{abstract}

\maketitle

\section{Introduction}

A study of graded Betti numbers of monomial ideals is a central topic in combinatorial commutative algebra.
Francisco, H\`a and Van Tuyl \cite{FHV} introduced an interesting method to study Betti numbers of monomial ideals, which they call {\em Betti splittings},
inspired from the work of Eliahou and Kervaire \cite{EK}.
In particular, 
a number of applications of Betti splittings for edge ideals of graphs were discovered in \cite{FHV,HV,V}.
On the other hand,
not much are known about Betti splittings of cover ideals, and it was asked by Van Tuyl \cite[Question 2.26 and Excercise 5.1.13]{V} if there are nice ways to construct Betti splittings of cover ideals.
In this short note, we find new Betti splittings of cover ideals and apply it to study multigraded Betti numbers of bipartite graphs.

We first define graded Betti numbers and cover ideals.
Let $S=\Bbbk[x_1,\dots,x_n]$ be a polynomial ring over a field $\Bbbk$.
We consider the standard grading of $S$ defined by $\deg (x_i)=1$ for any $i$
as well as the $\mathbb Z^n$-grading (multigrading) of $S$ defined by $\deg (x_i)=\mathbf e_i \in \mathbb Z^n$ for any $i$,
where $\mathbf e_1,\dots,\mathbf e_n$ are standard vectors of $\mathbb Z^n$.
For a homogeneous ideal $I \subset S$
and $i,j \in \mathbb Z_{\geq 0}$,
the number
\[
\beta_{i,j}(I)=\dim_\Bbbk \Tor_i(\Bbbk,I)_j
\]
is called the {\em $(i,j)$th graded Betti number of $I$},
where $M_j$ denotes the degree $j$ graded component of a graded $S$-module $M$.
If $I$ is a monomial ideal, then $I$ is a $\mathbb Z^n$-graded $S$-module and we can define the {\em $\mathbb Z^n$-graded Betti number}
\[
\beta_{i,\bm a}(I)=\dim_\Bbbk \Tor_i(\Bbbk,I)_{\bm a}
\]
for $i \in \mathbb Z_{\geq 0}$ and $\bm a \in \mathbb Z_{\geq 0}^n$.
We next recall cover ideals.
Let $G$ be a (finite simple) graph on $[n]=\{1,2,\dots,n\}$ with the edge set $E(G)$.
Thus $G$ is the pair $([n],E(G))$ of the set $[n]$ and a collection $E(G)$ of $2$-element subsets of $[n]$.
Elements of $[n]$ are called vertices of $G$
and elements of $E(G)$ are called edges of $G$.
A vertex $v$ of $G$ is {\em isolated} if $G$ has no edges containing $v$.
A subset $W \subset [n]$ is a {\em vertex cover} of $G$ if, for every edge $\{u,v \} \in E(G)$ one has $u \in W$ or $v \in W$.
A vertex cover of $G$ which is minimal with respect to inclusion is called a {\em minimal vertex cover} of $G$.
The {\em cover ideal} $J(G)$ of $G$ is the monomial ideal
\[
J(G)=(x^A \mid \mbox{$A$ is a (minimal) vertex cover of $G$}) \subset S,
\]
where $x^A=\prod_{i\in A} x_i$.
The cover ideal $J(G)$ is closely related to the the edge ideal $I(G)$ of $G$,
which is the monomial ideal generated by all $x_ix_j$ with $\{i,j\} \in E(G)$.
Indeed, it is known that $J(G)$ is the Alexander dual ideal of $I(G)$.
Studying algebraic invariants of edge ideals and cover ideals are one of the central topics in combinatorial commutative algebra.
See \cite{HH,V,Vi} for more information on this subject.

Next, we recall Betti splittings introduced in \cite{FHV}.
For a monomial ideal $I \subset S$,
we write $\mathcal G(I)$ for the set of minimal monomial generators of $I$.
We say that a pair $(K,L)$ of monomial ideals of $S$ is a {\em partition} of $I$ if $\mathcal G(I)=\mathcal G (K) \cup \mathcal G(L)$ and $\mathcal G(K) \cap \mathcal G(L)=\varnothing$.
We say that a partition $(K,L)$ of $I$ is a {\em Betti splitting} of $I$ if
\[
\beta_{i,j}(I)=\beta_{i,j}(K)+\beta_{i,j}(L)+\beta_{i-1,j}(K\cap L)
\ \ \mbox{ for all $i \geq 1$ and $j \in \mathbb Z_{\geq 0}$.}
\]
For convention we consider that a trivial partition $(I,(0))$ is also a Betti splitting.

We are actually interested in a special type of partitions called $x_i$-partitions.
Let $I$ be a monomial ideal and $\mathcal G_i=\{m \in \mathcal G(I) \mid \mbox{$x_i$ divides $m$}\}$.
Then $((\mathcal G_i),(\mathcal G(I) \setminus \mathcal G_i))$ is a partition of $I$ and is called the {\em $x_i$-partition} of $I$.
It was proved in \cite[Corollary 3.1]{FHV} that,
for edge ideals of graphs, $x_i$-partitions are always Betti splittings.
For cover ideals,
$x_i$-partitions are not always Betti splittings, but it is natural to ask when $x_i$-partitions give Betti splittings.
Our main result below is motivated by this natural question.
For a graph $G$ on $[n]$ and $v \in [n]$,
the set
\[
N_G(v)=\{u \in [n] \mid \{v,u\} \in E(G)\}
\]
is called the {\em neighbour} of $v$ in $G$.
We say that $W \subset [n]$ is an {\em independent set} of $G$ if $\{u,v\} \not \in E(G)$ for any $u,v \in W$.

\begin{theorem}
\label{thm:1-1}
Let $G$ be a simple graph on $[n]$ and $v \in [n]$.
If $N_G(v)$ is an independent set of $G$,
then the $x_v$-partition of $J(G)$ is a Betti splitting.
\end{theorem}

The assumption that $N_G(v)$ is an independent set is a little restrictive, but an advantage of the above theorem is that we can find a Betti splitting by only looking local information of a graph $G$ near the vertex $v$
without looking global information of $G$.
Also, this result generalizes \cite[Theorem 3.8]{FHV}, proving that a Cohen-Macaulay bipartite graph has a vertex $v$ such that the $x_v$-partition of $J(G)$ is a Betti splitting (see also Corollary \ref{cor}).
Since neighbours of vertices are always independent sets for bipartite graphs,
Theorem \ref{thm:1-1} actually tells the following.

\begin{corollary}
\label{cor:1-2}
If $G$ is a bipartite graph on $[n]$,
then the $x_v$-partition of $J(G)$ is a Betti splitting for any $v \in [n]$.
\end{corollary}

In addition to Corollary \ref{cor:1-2},
our proof of Theorem \ref{thm:1-1} gives another interesting result on multigraded Betti numbers of cover ideals of bipartite graphs.
Recall that a graph $G$ on $[n]$ is {\em bipartite} if there is a partition $[n]=U \cup W$ such that $E(G) \subset \{\{u,w\} \mid u \in U,w \in W\}$.
We call the above partition $[n]=U\cup W$ a {\em bipartition} of $G$.
For a bipartite graph $G$ with a bipartition $[n]=U \cup W$,
we call a squarefree monomial ideal
\[
M=(x^{N_G(u)} \mid u \in U) \subset \Bbbk[x_w \mid w \in W]
\]
a squarefree monomial ideal associated with $G$ (w.r.t.\ a bipartition $[n]=U\cup W$)\footnote{The ideal $M'=(x^{N_G(w)}\mid w \in W)$ is also a squarefree monomial ideal associated with $G$.}.

\begin{theorem}
\label{thm:1-3}
Let $G$ be a bipartite graph on $[n]$
with a bipartition $[n]=\{1,2,\dots,m\} \cup\{m+1,\dots,n\}$
and let $M=(x^{N_G(m+1)},\dots,x^{N_G(n)}) \subset \Bbbk[x_1,\dots,x_m]$ be a squarefree monomial ideal associated with $G$.
If $G$ has no isolated vertices,
then one has
\[
\beta_{i,(\bm a,1,\dots,1)}(J(G))=
\beta_{i-1,\bm a}(M)
\ \ \ \mbox{for any $i \geq 1$ and $\bm a \in \mathbb Z_{\geq 0}^m$.}\]
\end{theorem}

\begin{example}
\label{ex1.4}
Consider the bipartite graph $G$ below. \begin{center}
\includegraphics[scale=0.8]{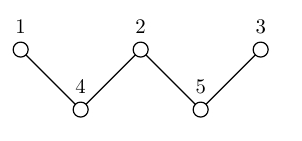}
\end{center}
Then 
\[
J(G)=(x_1x_2x_3,x_4x_5,x_1x_2x_5,x_2x_3x_4)
\ \ \mbox{ and } \  \ M=(x_1x_2,x_2x_3).\]
Multigraded Betti numbers of these ideals are given as follows:
{\footnotesize
\begin{align}
\label{0-0}
\beta_{i,\bm a}(J(G))
=
\begin{cases}
1 & \mbox{ if } (i,\bm a) \in 
 \left\{
\begin{array}{l}
\big(0,(1,1,1,0,0)\big),\big(0,(0,0,0,1,1)\big),\big(0,(1,1,0,0,1)\big),\\
\big(0,(0,1,1,1,0)\big),
\big(1,(1,1,1,1,0)\big),\big(1,(1,1,1,0,1)\big),\\ {\color{red}\big(1,(1,1,0,1,1)\big),
\big(1,(0,1,1,1,1)\big),\big(2,(1,1,1,1,1)\big)}
\end{array}
\right\},
\\
0 & \mbox{otherwise},
\end{cases}
\end{align}
}
and
\[
\beta_{i,\bm a}(M)
=
\begin{cases}
1 & \mbox{ if } (i,\bm a) \in 
 \left\{
\big(0,(1,1,0)\big),\big(0,(0,1,1)\big),\big(1,(1,1,1)\big)
\right\},
\\
0 & \mbox{otherwise}.
\end{cases}
\]
One can see that multigraded Betti numbers of $J(G)$ at the position $i$ and degree $(*,*,*,1,1)$ (see positions colored red in \eqref{0-0}) coincides with the multigraded Betti numbers of $M$ at the position $i-1$ and degree $(*,*,*)$.
\end{example}

Any squarefree monomial ideal can appear as a squarefree monomial ideal associated with a bipartite graph since the ideal
$M=(x^{A_1},\dots,x^{A_\ell}) \subset \Bbbk [x_1,\dots,x_m]$
is a squarefree monomial ideal associated with the bipartite graph $G$ on $[m+\ell]$ with $E(G)=\{\{j,m+i\} \mid i \in [\ell],\ j \in A_i\}$.
%Theorem \ref{thm:1-3} somehow tells that the multigraded Betti numbers of any squarefree monomial ideal can be embedded into multigraded Betti numbers of the cover ideal of a bipartite graph,
So the theorem tells that any strange behavior of multigraded Betti numbers of squarefree monomial ideals would also appear to multigraded Betti numbers of cover ideals of bipartite graphs.
We also note that Theorem \ref{thm:1-3} refines the result of Dalili and Kummini \cite[Theorem 4.7]{DK}, who proved that, with the same notation as in Theorem \ref{thm:1-3},
the Stanley--Reisner complex of the edge ideal of $G$ has the same homologies as the Alexander dual simplicial complex of $M$.
Via the Alexander duality and Hochster's formula (see e.g.\ \cite[Proposition 5.1.10 and Theorem 8.1.1]{HH}), this isomorphism of homologies can be considered as a special case of Theorem \ref{thm:1-3} when $\bm a=(1,1,\dots,1)$.

\section{A key proposition}

In this section, we prove a key proposition about multigraded Betti numbers of cover ideals, which will play a crucial role to prove Theorems \ref{thm:1-1} and \ref{thm:1-3}.

We actually consider splittings which are more general than $x_i$-partitions.
Let $G$ be a graph on $[n]$ and $U \subset [n]$.
We write
\[
J^U\!(G)=(x^W \mid W \mbox{ is a minimal vertex cover of $G$ with $W \supset U$})
\]
and
\[
J_U(G)=(x^W \mid W \mbox{ is a (minimal) vertex cover of $G$ with $W \not \supset U$}).
\]
We simply write $J^U\!(G)=J^v(G)$ and $J_U(G)=J_v(G)$ when $U=\{v\}$.
Clearly the pair $(J^U\!(G),J_U(G))$ is a splitting of $J(G)$.
Also, $(J^v(G),J_v(G))$ is nothing but the $x_v$-partition of $J(G)$.
We also write
\[
\tilde J^U\!(G)=(x^W \mid W \mbox{ is a vertex cover of $G$ with $W \supset U$}).
\]
If we write $G\setminus U$ for the graph obtained from $G$ by deleting all vertices  in $U$,
then we have
\[
\tilde J^U\!(G)=x^U \cdot J(G\setminus U),\]
where
we consider that $J(G\setminus U)$ is the ideal of $S$ not the ideal of $\Bbbk[x_v\mid v \in [n] \setminus U]$
and $J(G\setminus U)=S$ if $G\setminus U$ has no edges.
Note that 
$E(G\setminus U)=\{\{u,v\} \in E(G)\mid \{u,v\} \cap U= \varnothing\}$
and variables $x_u$ with $u \in U$ do not appear in $\mathcal G(J(G\setminus U))$.

We often use following easy facts on vertex covers.

\begin{lemma}
\label{lem:2-0}
Let $G$ be a graph on $[n]$,
$W$ a vertex cover of $G$
and $v \in [n]$.
\begin{itemize}
\item[(1)]
If $W \supset \{v\} \cup N_G(v)$ then $W \setminus \{v\}$ is also a vertex cover of $G$.
\item[(2)] If $W$ is a minimal vertex cover of $G$, then
$W$ does not contain $v$ if and only if $W \supset N_G(v)$.
\end{itemize}
\end{lemma}

The ideals $J^v(G)$ and $J_v(G)$ have the following another expressions.

\begin{lemma}
\label{lem:2-1}
Let $G$ be a graph on $[n]$,
$v \in [n]$ and $N=N_G(v)$.
Then
\begin{enumerate}
\item $J^v(G)=J_N(G)=x_v \cdot J_N(G\setminus \{v\})$.
\item $J_v(G)=J^N\!(G)=\tilde J^N\!(G)$.
\end{enumerate}
\end{lemma}

\begin{proof}
The equations $J^v(G)=J_N(G)$ and $J_v(G)=J^N\!(G)$ immediately follow from Lemma \ref{lem:2-0}(2).

We have $J^v(G) \subset x_v\cdot J_N(G\setminus \{v\})$ since, for any minimal vertex cover $W$ of $G$ with $v \in W$,
one has $W \not \supset N$ by Lemma \ref{lem:2-0}(2)
and $W\setminus \{v\}$ is a vertex cover of $G\setminus \{v\}$.
Also, we have $J_N(G) \supset x_v\cdot J_N(G\setminus \{v\})$ since if $W \subset [n] \setminus \{v\}$ is a vertex cover of $G\setminus \{v\}$ with $W\not \supset N$,
then $\{v\} \cup W$ is a vertex cover of $G$ that does not contain $N$.
These prove (1).

It is clear that $J^N\!(G) \subset \tilde J^N\!(G)=x^N \cdot J(G \setminus N)$, so to prove (2) it suffices to prove $x^N \cdot J(G \setminus N) \subset J_v(G)$.
But if $W$ is a minimal vertex cover of $G\setminus N$,
then $v \not \in W$ since $v$ is an isolated vertex of $G\setminus N$ and $N \cup W$ is a vertex cover of $G$.
This proves the desired inclusion.
\end{proof}

We now prove our key proposition.
For $\bm a=(a_1,\dots,a_n),
\bm b=(b_1,\dots,b_n) \in \mathbb Z^n$,
we write $\bm a \geq \bm b$ if $a_i \geq b_i$ for all $i$.
Also, for a subset $W \subset [n]$,
we write $\mathbf e_W=\sum_{i\in W} \mathbf e_i$.

\begin{proposition}
\label{prop:2-2}
Let $G$ be a graph on $[n]$ and $U \subset [n]$ an independent set of $G$.
Then one has
\[
\Tor_i(\Bbbk,J_U(G))_{\bm a}=0
\ \ \mbox{ for all }i \in \mathbb Z_{\geq 0} \mbox{ and } \bm a \in \mathbb Z^n_{\geq 0} \mbox{ with }\bm a \geq \mathbf e_U.
\]
\end{proposition}

\begin{proof}
Without loss of generality we may assume $U=\{1,2,\dots,\ell\}$ for some $\ell \in [n]$.
For any $T \subset U$, we define
\[
I_T=( x^W \mid \mbox{ $W$ is a (minimal) vertex cover of $G$ with $W \subset [n]\setminus T$}).
\]
We note that
$J_U(G)=I_{\{1\}}+I_{\{2\}}+ \cdots + I_{\{\ell\}}.$

We first claim
\begin{align}
\label{eq:2.2}
I_{\{u\}} \cap I_{\{v\}}=I_{\{u,v\}}
\ \ \mbox{ for }u,v \in U.
\end{align}
The inclusion ``$\supset$" in \eqref{eq:2.2} is clear.
We prove the opposite inclusion.
Let $x^A \in \mathcal G(I_{\{u\}})$ and $x^B \in \mathcal G(I_{\{v\}})$.
What we must prove is
$\mathrm{lcm}(x^A,x^B) \in I_{\{u,v\}}$.
Since $A$ and $B$ are vertex covers of $G$ with $u \not \in A$ and $v \not \in B$,
we have
\[
N_G(u) \subset A \mbox{ and } N_G(v) \subset B.
\]
Since $U$ is an independent set of $G$ we have $\{u,v\} \cap (N_G(u)\cup N_G(v))=\varnothing$, so the above inclusions and Lemma \ref{lem:2-0}(1) tell
$C=(A\cup B) \setminus \{u,v\}$ is a vertex cover of $G$.
(We note that $A\cup B$ may not contain $u$ or $v$.)
This tells $\mathrm{lcm}(x^A,x^B)=x^{A \cup B} \in I_{\{u,v\}}$
since it is divisible by $x^C \in I_{\{u,v\}}$.

We now prove the desired vanishing statement for $\Tor_i(\Bbbk,J_U(G))$.
Let
\[
L_{\{k\}}=I_{\{1\}}+\cdots+I_{\{k\}}
\]
for $k=1,2,\dots,\ell$.
We claim that
\begin{align}
\label{eq:2-3}
\Tor_i(\Bbbk,L_{\{k\}})_{\bm a}=0
\ \ \mbox{ for all }i \in \mathbb Z_{\geq 0} \mbox{ and } \bm a \in \mathbb Z^n_{\geq 0} \mbox{ with } \bm a \geq \mathbf e_1 + \cdots + \mathbf e_k.
\end{align}
Note that $L_{\{\ell\}}=J_U(G)$ so \eqref{eq:2-3} proves the proposition.

We prove \eqref{eq:2-3} by induction on $k$.
The case when $k=1$ is obvious since $L_{\{1\}}=I_{\{1\}}$ contains no generators which are divisible by $x_1$.
Suppose $k \geq 2$.
The short exact sequence 
$0 \to L_{\{k-1\}} \cap I_{\{k\}} \to L_{\{k-1\}} \oplus I_{\{k\}} \to L_{\{k\}} \to 0$
induces the long exact sequence
{\small
\begin{align*}
%\label{eq:long}
\begin{array}{l}
\cdots\!
\to \Tor_i(\Bbbk,L_{\{k-1\}})_{\bm a}
\!\oplus\!
\Tor_i(\Bbbk,I_{\{k\}})_{\bm a}
\!\to \!
\Tor_i(\Bbbk,L_{\{k\}})_{\bm a}
\!\to\!
\Tor_{i-1}(\Bbbk,L_{\{k-1\}}\!\cap\! I_{\{k\}}
)_{\bm a}
\!\to\! \cdots
\end{array}
\end{align*}
}
\hspace{-6pt}
for any $\bm a \in \mathbb Z^n_{\geq 0}$.
By the induction hypothesis, we have
\begin{align}
\label{eq:2.4}
\Tor_i(\Bbbk,L_{\{k-1\}})_{\bm a}=0 \ \ 
\mbox{ 
%for all $i \in \mathbb Z_{\geq 0}$ and 
if $\bm a \geq \mathbf e_1 + \cdots + \mathbf e_{k-1}$}.
\end{align}
Also, \eqref{eq:2.2} tells 
$L_{\{k-1\}} \cap I_{\{k\}}=\sum_{j=1}^{k-1} (I_{\{j\}} \cap I_{\{k\}})=\sum_{j=1}^{k-1}I_{\{j,k\}}$ has no generators which is divisible by $x_k$, so we have
\begin{align}
\label{eq:2.5}
\Tor_i(\Bbbk,L_{\{k-1\}} \cap I_{\{k\}})_{\bm a}=\Tor_i(\Bbbk,I_{\{k\}})_{\bm a}=0 \ \ 
\mbox{
% for all $i \in \mathbb Z_{\geq 0}$ and 
if $\bm a \geq \mathbf e_k$}.
\end{align}
Applying \eqref{eq:2.4} and \eqref{eq:2.5} to the long exact sequence of $\Tor$,
we get the desired property \eqref{eq:2-3}.
\end{proof}

To consider Betti splittings of the form $(J^U\!(G),J_U(G))$,
it is convenient to know the ideal $J^U\!(G)\cap J_U(G)$.
While we do not have a simple formula for this intersection in general,
we close this section by showing the following formula for $\tilde J^U\!(G) \cap J_U(G)$.

\begin{lemma}
\label{lem:2-3}
Let $G$ be a graph on $[n]$ and $U \subset [n]$.
Then
\begin{align}
\label{eq:2-7}
(x^U) \cap J_U(G)=\tilde J^U\!(G) \cap J_U(G)= \sum_{u \in U}x^{U \cup N_G(u)} \cdot J\big(G\setminus (U \cup N_G(u))\big).
\end{align}
\end{lemma}

\begin{proof}
It is clear that the first term contains the second term in \eqref{eq:2-7}.
Also, if $W \subset [n] \setminus (U \cup N_G(u))$ is a vertex cover of $G \setminus (U \cup N_G(u))$ then $W'=W \cup U \cup N_G(u)$ is a vertex cover of $G$ containing $U$ and by Lemma \ref{lem:2-0}(1) the set $W' \setminus \{u\}$ is also a vertex cover of $G$.
This proves that the second term contains the third term in \eqref{eq:2-7}.
We prove that the first term is contained in the third term in \eqref{eq:2-7}.
Let $x^A =x^U \cdot x^W \in (x^U) \cap J_U(G)$.
Then, by the definition of $J_U(G)$,
there is a $u \in U$ such that $x^{A \setminus \{u\}}=x^{U \setminus \{u\}} x^W \in J(G)$.
Since $A \setminus \{u\}$ is a vertex cover of $G$, we have $N_G(u) \subset A \setminus \{u\}$ and it follows that $A \setminus (U \cup N_G(u))$ is a vertex cover of $G\setminus (U \cup N_G(u))$.
This completes the proof.
\end{proof}

\begin{remark}
\label{rem:2-4}
Another expression of the first term of \eqref{eq:2-7} is
\[
(x^U) \cap J_U(G)=\left(x^{U \cup W}\mid \mbox{ $W$ is a vertex cover of $G$ with $W \not \supset U$}\right).
\]
\end{remark}

\begin{remark}
\label{rem:2-5}
If $G$ is a bipartite graph on $[n]$ with a bipartition $[n]=U \cup W$
and if no elements in $U$ are isolated vertices, 
then $J^U\!(G)=(x^U)$ and Lemma \ref{lem:2-3} tells
\[
J^U\!(G)\cap J_U(G)=x^U \cdot \left (x^{N_G(u)} \mid u \in U \right)
\]
since $G\setminus U$ has no edges (which implies $J(G\setminus (U\cup N_G(u)))=S$ for any $u \in U$).
The ideal $(x^{N_G(u)} \mid u \in U)$ appearing in the above equation is essentially a monomial ideal associated with $G$ given in the introduction.
\end{remark}

\section{Two Betti splittings}

In this last section, we prove Theorems \ref{thm:1-1} and \ref{thm:1-3}.

We first recall a criterion of Betti splittings,
which we will use.
Let $(K,L)$ be a partition of a monomial ideal $J \subset S$.
Then the short exact sequence $0 \to K \cap L \to K \oplus L \to J \to 0$ induces the long exact sequence
\begin{align}
\label{eq:3.1}
\cdots
\to \Tor_i(\Bbbk,K \cap L)
\stackrel {\varphi_i} \longrightarrow \Tor_i(\Bbbk,K )\oplus \Tor_i(\Bbbk,L) 
\to \Tor_i(\Bbbk,J) \to \cdots.
\end{align}
A Betti splitting can be characterized as follows.

\begin{lemma}[{\cite[Proposition 2.1]{FHV}}]
\label{lem:3-1}
With the same notation as above,
$(K,L)$ is a Betti splitting if and only if $\varphi_i$ is the zero map for all $i$.
\end{lemma}

We also need the following fact.

\begin{lemma}
\label{lem:3-2}
Let $G$ be a graph on $[n]$ and $U \subset [n]$.
Then
\[
\Tor_i(\Bbbk,J(G))_{\bm a} \cong \Tor_i(\Bbbk,J_U(G))_{\bm a}
\ \  \mbox{ for all } i \in \mathbb Z_{\geq 0} \mbox{ and }\bm a \in \mathbb Z^n_{\geq 0} \mbox{ with } \bm a \not \geq \mathbf e_U.
\]
\end{lemma}

\begin{proof}
The short exact sequence
\[
0 \to J^U\!(G) \cap J_U(G) \to J^U\!(G) \oplus J_U(G) \to J(G)\to 0\]
induces the long exact sequence
{\small
\[
\cdots\!
\to\! \Tor_i(\Bbbk,J^U\!(G)\! \cap\! J_U(G))
\! \to\! \Tor_i(\Bbbk,J^U\!(G) )\!\oplus\! \Tor_i(\Bbbk,J_U(G)) 
\!\to\! \Tor_i(\Bbbk,J(G)) \!\to\! \cdots.
\]
}
\hspace{-6pt}
Since generators of $J^U\!(G) \cap J_U(G)$ and $J^U\!(G)$ are divisible by $x^U$ we have
\[
\Tor_i(\Bbbk,J^U\!(G) \cap J_U(G))_{\bm a} \cong
\Tor_i(\Bbbk,J^U\!(G))_{\bm a}=0
\ \ \mbox{ if } \bm a \not \geq \mathbf e_U.
\]
This and the above long exact sequence prove the desired statement.
\end{proof}

We now prove Theorem \ref{thm:1-1} in the introduction.

\begin{theorem}
\label{thm:3-3}
Let $G$ be a simple graph on $[n]$ and $v \in [n]$.
Suppose that $N=N_G(v)$ is an independent set of $G$.
Then $(J^v(G),J_v(G))$ is a Betti splitting.
Moreover, for any $i \geq 1$ and $\bm a \in \mathbb Z^n_{\geq 0}$, we have
\begin{align*}
\beta_{i,\bm a}(J(G))=
\begin{cases}
\beta_{i-1,\bm a}\big(\sum_{u \in N} x^{N \cup N_G(u)}\! \cdot \! J(G \!\setminus\! (N \cup N_G(u)))\big) & \mbox { if } \bm a \geq \mathbf e_v+\mathbf e_{N},\\
\beta_{i,\bm a}(x_v \cdot J(G \setminus \{v\})) & \mbox{ if $\bm a \geq \mathbf e_v$ and $\bm a \not \geq \mathbf e_N,$}\\
\beta_{i,\bm a}(x^N \cdot J(G\setminus N))
& \mbox{ if  $\bm a \not \geq \mathbf e_v$}. 
\end{cases}
\end{align*}
\end{theorem}

\begin{proof}
Let $J=J(G)$, $K=J^v(G)$ and $L=J_v(G)$.
By Lemmas \ref{lem:2-1} and \ref{lem:2-3}, we have
\begin{align}
\label{eq:3-2}
K=J_N(G)=x_v \cdot J_N(G\setminus \{v\}),
\end{align}
\begin{align}
\label{eq:3-3}
L=\tilde J^N\!(G)=x^N \cdot J(G\setminus N),
\end{align}
and
\begin{align}
\label{eq:3-4}
\textstyle
K \cap L=J_N(G)\cap \tilde J_N\!(G)=\sum_{u \in N} x^{N \cup N_G(u)} J(G \setminus (N \cup N_G(u))).
\end{align}
Considering the long exact sequence \eqref{eq:3.1} for the pair $(K,L)=(J^v(G),J_v(G))$,
we have
{\small
\begin{align}
\label{eq:3.1.1}
\cdots 
\stackrel{\varphi_i}\longrightarrow \Tor_i(\Bbbk,K) _{\bm a}
\oplus
\Tor_i(\Bbbk,L) _{\bm a}
\to \Tor_i(\Bbbk,J) _{\bm a}
\to  \Tor_{i-1}(\Bbbk,K\cap L) _{\bm a} 
\stackrel{\varphi_{i-1} }\longrightarrow \cdots
\end{align}
}
\hspace{-8pt}
for any $\bm a \in \mathbb Z^n_{\geq 0}$.
We prove that $\varphi_i$ is the zero map for all $i$,
which proves the first statement by Lemma \ref{lem:3-1}.

Since any generator of $K \cap L$
is divisible by $x_vx^N$, we have
\begin{align}
\label{eq:3.6}
\Tor_i(\Bbbk, K \cap L)_{\bm a}=0
\ \ \mbox{ if }\bm a \not \geq \mathbf e_v + \mathbf e_N.
\end{align}
Since $L=J_v(G)$ does not contain any generator which is divisible by $x_v$,
we have
\begin{align}
\label{eq:3.7}
\Tor_i(\Bbbk, L)_{\bm a}=0
\ \ \mbox{ if }\bm a \geq \mathbf e_v.
\end{align}
Also, \eqref{eq:3-2} and Proposition \ref{prop:2-2} tell
\begin{align}
\label{eq:3.8}
\Tor_i(\Bbbk, K)_{\bm a}=0
\ \ \mbox{ if }\bm a \not \geq \mathbf e_v \mbox{ or } \bm a \geq \mathbf e_N.
\end{align}
Applying \eqref{eq:3.6}, \eqref{eq:3.7} and \eqref{eq:3.8} to the long exact sequence \eqref{eq:3.1.1},
we can see that $\varphi_i$ is the zero map,
proving that $(J^v(G),J_v(G))$ is a Betti splitting,
and
\begin{align*}
\Tor_i(\Bbbk,J(G))_{\bm a}
\cong
\begin{cases}
\Tor_{i-1}(\Bbbk,K \cap L)_{\bm a} & \mbox{ if }\bm a \geq \mathbf e_v + \mathbf e_N,\\
\Tor_{i}(\Bbbk,K)_{\bm a} & \mbox{ if }\bm a \geq \mathbf e_v \mbox{ and } \bm a \not \geq \mathbf e_N,\\
\Tor_{i}(\Bbbk,L)_{\bm a} & \mbox{ if }\bm a \not \geq \mathbf e_v.
\end{cases}
\end{align*}
This isomorphism proves the desired formula of $\beta_{i,\bm a}(J(G))$ using \eqref{eq:3-2}, \eqref{eq:3-3}, \eqref{eq:3-4} and Lemma \ref{lem:3-2}.
\end{proof}

\begin{example}
The assumption that $N_G(v)$ is an independent set is necessary in Theorem \ref{thm:3-3}.
Indeed, if we consider the graph $G$ below, then according to the computation by Macaulay2 \cite{GS}
we have $\beta_{1,6}(J^v(G)) =1$ but $\beta_{1,6}(J(G))=0$,
so $(J^v(G),J_v(G))$ is not a Betti splitting of $J(G)$.
\begin{center}
\includegraphics[scale=1.2]{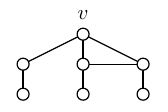}
\end{center}\end{example}

\begin{example}
Consider the graph $G$ in Example \ref{ex1.4}
and let $v=4$.
Then $N=N_G(4)=\{1,2\}$. Also
by \eqref{eq:3-2}, \eqref{eq:3-3} and \eqref{eq:3-4} we have
\[
x_v \cdot J(G\setminus \{v\})=(x_4x_5,x_2x_3x_4),
\ \
x^N\cdot J(G\setminus N)=(x_1x_2x_3,x_1x_2x_5)
\]
and
\begin{align*}
J^v(G)\cap J_v(G)
&=
x_1x_2x_4 \cdot J(G\setminus \{1,2,4\})+
x_1x_2x_4x_5\cdot J(G\setminus \{1,2,4,5\})\\
&=(x_1x_2x_3x_4,x_1x_2x_4x_5).
\end{align*}
It is easy to see that
{\footnotesize
\[
\beta_{i,\bm a}(x_v \cdot J(G\setminus \{v\}))=1
\mbox{ when }
(i,\bm a)\in \{(0,(0,0,0,1,1)),(0,(0,1,1,1,0)),(1,(0,1,1,1,1))\},
\]
\[
\beta_{i,\bm a}(x^N \cdot J(G\setminus N))=1
\mbox{ when }
(i,\bm a)\in \{(0,(1,1,1,0,0)),(0,(1,1,0,0,1)),(1,(1,1,1,0,1))\},
\smallskip\]
\[
\beta_{i,\bm a}(J^v(G)\cap J_v(G)
)=1
\mbox{ when }
(i,\bm a)\in \{(0,(1,1,1,1,0)),(0,(1,1,0,1,1)),(1,(1,1,1,1,1))\}
\]
}
\hspace{-8pt}
and all other values of $\beta_{i,\bm a}$ are zero.
Compare these computations with the multigraded Betti numbers of $J(G)$ given in Example \ref{ex1.4}.

We note that in this case we have $J^{\{4\}}(G)=x_4\cdot J(G\setminus \{4\})$,
but generally $J^{v}(G)$ is not equal to $x_v\cdot J(G\setminus \{v\})$.
For example $J^{\{2\}}(G)=x_2(x_1x_3,x_1x_5,x_3x_4)$ but $J(G\setminus \{2\})=(x_1x_3,x_1x_5,x_3x_4,x_4x_5)$.
\end{example}

Francisco, H\`a and Van Tuyl \cite[Theorem 3.8]{FHV} found a simple inductive formula of graded Betti numbers of cover ideals of Cohen--Macaulay bipartite graphs using Betti splittings.
This result can be considered as a special case of Theorem \ref{thm:3-3}.
Indeed, the proof of \cite[Theorem 3.8]{FHV} tells that if $G$ is a Cohen--Macaulay bipartite graph then $G$ must have a leaf vertex $v$,
this is, $N_G(v)=\{u\}$ for some vertex $u$,
and the $x_v$-partition of $J(G)$ is a Betti splitting.
In this case $N_G(v)$ is automatically an independent set of $G$ and we have
\[
J^v(G)=x_v\cdot J_u(G\setminus\{v\})=x^{N_G(u)}\cdot J(G\setminus N_G(u)),\ \
J_v(G)=x_u \cdot J(G \setminus \{u,v\})\]
and
\[
J^v(G) \cap J_v(G)=x_ux^{N_G(u)} \cdot J(G\setminus N_G(u)).
\]
The above equations actually hold whenever $N_G(v)=\{u\}$, so
by Theorem \ref{thm:3-3} we get the following statement.

\begin{corollary}[Francisco-H\`a-Van Tuyl]
\label{cor}
Let $G$ be a graph on $[n]$.
Suppose that $G$ has a vertex $v$ with $N_G(v)=\{u\}$ for some $u$. Let $G'=G\setminus \{u,v\}$ and $G''=G\setminus N_G(u)$. Then, for all $i \geq 1$ and $j \in \mathbb Z_{\geq 0}$ one has
\[
\beta_{i,j}(J(G))=
\beta_{i,j-1}(J(G'))
+\beta_{i,j-|N_G(u)|}(J(G''))+
\beta_{i-1,j-|N_G(u)|-1}(J(G'')).
\]
\end{corollary}

\begin{remark}
It was recently proved in \cite[Proposition 2.3]{CF} that any Cohen--Macaulay very well-covered graph $G$ admit a vertex $v$ such that the $x_v$-partition of $J(G)$ is a Betti splitting. Theorem \ref{thm:3-3} does not cover this result.
\end{remark}

We finally prove Theorem \ref{thm:1-3}.

\begin{theorem}
\label{thm:3.4}
Let $G$ be a bipartite graph on $[n]$ with a bipartition $[n]=U \cup W$,
and let $M=(x^{N_G(u)}\mid u \in U) \subset S$.
Assume that no vertices in $U$ are isolated vertices.
Then $(J^{U}(G)=(x^{U}),J_{U}(G))$ is a Betti splitting of $J(G)$.
Moreover, for any $i \geq 1$ and $\bm a \in \mathbb Z^n_{\geq 0}$, one has
\begin{align}
\label{fin}
\beta_{i,\bm a}(J(G))=
\begin{cases}
\beta_{i,\bm a}(J_{U}(G)) & \mbox{ if } \bm a \not \geq \mathbf e_{U},\\
\beta_{i-1,\bm a}(x^U \cdot M) & \mbox{ if } \bm a  \geq \mathbf e_{U}.
\end{cases}
\end{align}
\end{theorem}

\begin{proof}
We note that $U$ is an independent set of $G$ and $x^U \cdot M=(x^U) \cap J_U(G)$ as we explained in Remark \ref{rem:2-5}.
Considering the long exact sequence \eqref{eq:3.1} for the pair $(J^U\!(G),J_U(G))$, we have
\[
\cdots \to \Tor_i(\Bbbk,x^U\cdot M)_{\bm a} \to \Tor_i(\Bbbk,(x^U))_{\bm a} \oplus \Tor_i(\Bbbk,J_U(G))_{\bm a} \to \Tor_i(\Bbbk,J(G))_{\bm a}\to \cdots
\]
for any $\bm a \in \mathbb Z^n_{\geq 0}$.
By Proposition \ref{prop:2-2},
\[\Tor_i(\Bbbk,J_U(G))_{\bm a}=0 \mbox{ for $i\geq 0$ and $\bm a \geq \mathbf e_U$}.\]
Then, since $\Tor_i(\Bbbk,(x^U))_{\bm a}$ is only non-zero when $(i,\bm a)=(0,\mathbf e_U)$ and since $M \ne S$,
we have
\[
\Tor_i(\Bbbk,J(G))_{\bm a} \cong \Tor_{i-1}(\Bbbk,x^U \cdot M)_{\bm a} \mbox{ for $i \geq 1$ and $\bm a \geq \mathbf e_U$}
\]
proving the desired equation \eqref{fin} when $\bm a \geq \mathbf e_U$.
The equation \eqref{fin} when $\bm a \not \geq \mathbf e_U$ is Lemma \ref{lem:3-2}.
\end{proof}
%\noindent
%\textbf{Acknowledgments}:

\end{document}